%% file: main.tex
\documentclass[11pt]{amsart}

\usepackage{amsmath,amssymb,amsthm,mathtools}
\usepackage{tikz-cd}
\usepackage{enumitem}
\usepackage{hyperref}
\usepackage[margin=1.15in]{geometry}

\numberwithin{equation}{section}

\newtheorem{theorem}{Theorem}[section]
\newtheorem{proposition}[theorem]{Proposition}
\newtheorem{lemma}[theorem]{Lemma}
\newtheorem{corollary}[theorem]{Corollary}

\theoremstyle{definition}

\newtheorem{remark}[theorem]{Remark}

\DeclareMathOperator{\Proj}{Proj}

\DeclareMathOperator{\ord}{ord}

\DeclareMathOperator{\Tr}{Tr}

\newcommand{\Fp}{\mathbb F_p}

\newcommand{\calS}{\mathcal S}
\newcommand{\ttt}{\mathbf t}
\newcommand{\aideal}{\mathfrak a}
\newcommand{\LL}{\mathbb L}

\title[Resolutions of \(p\)-cyclic quotients]
{Resolutions of Linear \(p\)-Cyclic Quotient Singularities}

\author[L. Fan]{Linghu Fan}
\address{Kavli Institute for the Physics and Mathematics of the Universe (WPI), The University of Tokyo, 5-1-5 Kashiwanoha,
Kashiwa-shi, Chiba 277-8561, Japan}
\email{linghu.fan@ipmu.jp}

\author[H. Li]{Hongmin Li}
\address{School of Life Science and Technology, Institute of Science Tokyo,
2-12-1 Ookayama, Meguro-ku, Tokyo 152-8550, Japan}
\address{Department of Computational Biology and Medical Sciences, Graduate
School of Frontier Sciences, The University of Tokyo, 5-1-5 Kashiwanoha,
Kashiwa-shi, Chiba 277-8561, Japan}
\email{lihongmin@edu.k.u-tokyo.ac.jp}

\date{\today}

\begin{document}

\begin{abstract}
Let \(k\) be an algebraically closed field of positive characteristic \(p\), and
let \(C_p\) act linearly on a finite-dimensional \(k\)-vector space \(V\), with indecomposable Jordan summands \(V_{d_i}\) of dimension \(d_i\).
We study projective (crepant) resolutions of \(V/C_p\) by constructing a model using the invariant weighted blowup.  For \(V=V_2^{\oplus n}\), our model
is smooth and is the normalization of a blowup of \(V/C_p\).  It has a unique
exceptional divisor of discrepancy \(n-p\), such that our model is a crepant resolution when \(n=p\). In almost all other cases when the quotient is canonical but not terminal, the invariant weighted blowup model does not produce crepant resolutions, but gives the unique nontrivial projective crepant birational model of the quotient. Consequently, up to trivial summands, we classify the \(p\)-cyclic linear quotients admitting a projective crepant resolution.
\end{abstract}



\maketitle

\input{pcyclic-sections/introduction-and-setup}
\input{pcyclic-sections/all-v2-family}

\input{pcyclic-sections/higher-block-obstruction}

\input{pcyclic-sections/boundary-classification}

\section*{Acknowledgements}
This research used the
\href{https://utelecon.adm.u-tokyo.ac.jp/en/research_computing/utokyo_azure/}{UTokyo Azure service}.
The authors also acknowledge support from OpenAI through the ChatGPT for
Academic Researchers program. Linghu Fan thanks Professor Takehiko Yasuda for valuable discussions. Linghu Fan was supported by JSPS KAKENHI Grant Number JP25KJ1153, and World Premier International Research Center Initiative (WPI), MEXT, Japan. Hongmin Li was supported by the 2025 Google Research Grant from Google Asia
Pacific Pte.\ Ltd.

\section*{Use of generative AI}
During this work, Hongmin Li used ChatGPT (OpenAI) to assist with
exploratory calculations in specific examples, the exploration of
mathematical arguments and possible proof strategies, and the organization
of working notes into draft proofs. GPT-5.5 was the principal model used
for the exploration and development of these proof drafts.
These drafts were subsequently
discussed, assessed, and refined by the authors, incorporating Linghu Fan's
mathematical feedback and correction. ChatGPT was also used to assist with revisions to
the exposition. Both authors revised the manuscript and take full
responsibility for its mathematical content and final presentation.

\bibliographystyle{alpha}
\bibliography{bibliography/refs}

\end{document}

%% file: pcyclic-sections/introduction-and-setup.tex
\section{Introduction}\label{sec:pcyclic-introduction}

Let \(k\) be an algebraically closed field of prime characteristic \(p\), and
let
\[
  G=C_p=\langle\sigma\rangle
\]
act linearly on a finite-dimensional \(k\)-vector space \(V\).  If
\(J_d(\lambda)\) denotes the Jordan block of size \(d\) and eigenvalue
\(\lambda\), then by abuse of notation, every indecomposable \(kG\)-module is
\[
  V_d:=J_d(1),\qquad 1\le d\le p.
\]
Thus, every representation space of \(C_p\) can be written as
\[
  V\simeq\bigoplus_{i=1}^s V_{d_i}.
\]
We additionally define
\[
  D_V:=\sum_{i=1}^s\binom{d_i}{2}
\]
according to Yasuda's \(p\)-cyclic McKay correspondence~\cite{Yasuda2014pCyclic}. The summands \(V_1\) are trivial and make no contribution to \(D_V\).

\(D_V\) can be used to classify the quotient singularities from the perspective of the minimal model theory.  For a \(p\)-cyclic quotient without pseudo-reflections, the quotient is log canonical, canonical, or terminal precisely
when
\[
  D_V\ge p-1,\qquad D_V\ge p,\qquad D_V>p,
\]
respectively~\cite[Corollary~1.4]{Yasuda2019Discrepancies}.  In this paper, we consider the
canonical non-terminal boundary condition \(D_V=p\) and ask when the quotient
\[
  X:=V/G
\]
admits a projective crepant resolution.

To study such \(p\)-cyclic quotients, we construct a natural model \(W\) by taking the \(C_p\)-invariant part of a natural weighted blowup of the representation space. The first result treats the direct sum of two-dimensional blocks.  In this
case we consider the ordinary blowup of the fixed locus on
the representation space \(V\), and its quotient can be described through polynomial affine charts.

\begin{theorem}\label{thm:all-v2}
Let \(n\ge2\), \(V=V_2^{\oplus n}\), and \(R=k[V]^{G}\).  Let
\(J\subset R\) be the ideal generated by the invariant linear coordinates
that cut out the fixed locus on \(V\).  Then the natural model
\[
  W
  \simeq \operatorname{Norm}\operatorname{Bl}_J(V/G).
\]
The induced morphism
\[
  f\colon W\longrightarrow V/G
\]
is projective and birational, and \(W\) is smooth.  Here \(\dim V=2n\).  The
morphism \(f\) has a unique exceptional prime \(F\), and
\[
  K_W=f^*K_{V/G}+(n-p)F.
\]
Moreover, we can compute the Euler characteristics
\[
  \chi(W)=\chi(F)=n.
\]
In particular, when \(n=p\), \(f\) is a projective crepant resolution of
\(V_2^{\oplus p}/G\).
\end{theorem}

When \(n=p\), the calculation of Euler characteristics here is compatible
with the \(p\)-cyclic McKay correspondence and with the explicit calculations
in~\cite{Yasuda2014pCyclic,Fan2024EulerCrepant}.  On the exceptional divisor,
the degeneration of the constant \(C_p\)-action is related to the linear
\(\alpha_p\)-quotients studied by Posva, R\"osler, and
Yasuda~\cite{PosvaRoslerYasuda2026AlphaP}.

Assume now that \(p\ge5\), \(D_V=p\), and \(V\) contains a Jordan block
of size at least \(3\).  The natural model \(W\) is terminal and crepant,
but a square-root chart exhibits a \(\mu_2\)-quotient singularity on \(W\).
We identify \(W\) with the relative Proj of the valuation algebra associated
with the unique exceptional divisorial valuation of discrepancy zero over
\(X\).  This identifies every nontrivial projective crepant birational
model of \(X\) with \(W\).  Since \(W\) is singular, \(X\) admits no
projective crepant resolution.  More precisely, we prove the following theorem.

\begin{theorem}[=Theorem~\ref{thm:hb-rigidity}]\label{thm:higher-block}
Assume \(p\ge5\) and \(D_V=p\).  If \(V\) contains a summand \(V_d\) with
\(d\ge3\), then the natural model
\[
  W
  \longrightarrow X=V/G
\]
is projective, crepant, terminal, and singular.  It is the unique nontrivial
projective crepant birational model of \(X\).  Consequently, \(X\) has no
projective crepant resolution.
\end{theorem}

When \(p=3\) and \(V=V_3\), we can obtain the full list of projective crepant birational models of \(V/G\) by the same approach, and there is only one smooth model among them, which coincides with the crepant resolution constructed in \cite{Yasuda2014pCyclic}. Combining the two main theorems together with this observation, we have the following corollary.

\begin{corollary}\label{cor:boundary-classification}
Assume that \(G\) has no pseudo-reflections. \(V/G\) admits a projective
crepant resolution if and only if after removing the trivial summands, one of the following
holds:
\begin{enumerate}[label=\textup{(\roman*)}]
  \item \(V\simeq V_2^{\oplus p}\). 
  \item \(p=3\) and \(V\simeq V_3\). 
\end{enumerate}
Furthermore, in both cases the projective crepant resolution is unique up to isomorphism over \(V/G\).
\end{corollary}

This paper is organized as follows. In Chapter~\ref{sec:setup}, we give the basic construction of our model \(W\) and some conventions in this paper. In Chapter~\ref{sec:all-v2-family}, we prove Theorem~\ref{thm:all-v2}. In Chapter~\ref{sec:higher-blocks}, we consider the general case, for which Chapter~\ref{sec:all-v2-family} serves as a guiding example. We also study the obstruction when \(D_V=p\ge5\) and there is a Jordan block of size at least \(3\) in this chapter. In Chapter~\ref{sec:pcyclic-classification}, we classify the linear \(p\)-cyclic quotient singularities that admit projective crepant resolutions.

\section{Basic construction and conventions}\label{sec:setup}

Retain the notation
\[
  G=C_p=\langle\sigma\rangle,
  \qquad
  V\simeq\bigoplus_{i=1}^s J_{d_i}(1)
  =\bigoplus_{i=1}^s V_{d_i},
  \qquad
  D_V=\sum_{i=1}^s\binom{d_i}{2},
\]
and set
\[
  N:=\dim V=\sum_{i=1}^s d_i.
\]
We only consider the case when \(G\) has no pseudo-reflections in this paper. It is equivalent to say that \(V\not\simeq V_2\) after removing the trivial summands, or \(D_V\neq 1\).
For the block \(V_{d_i}\), choose coordinate functions
\[
  x_{i,1},\ldots,x_{i,d_i}
\]
such that
\[
  \sigma(x_{i,j})=x_{i,j}+x_{i,j-1},
  \qquad x_{i,0}=0.
\]
Assign the Jordan weights
\[
  \operatorname{wt}(x_{i,j})=d_i-j.
\]
The positive weights in \(V_{d_i}\) are \(1,\ldots,d_i-1\), while every
coordinate belonging to a trivial block \(V_1\) has weight zero.

The coordinate rings and quotient varieties are
\[
  S:=k[V],\qquad R:=S^G,
  \qquad V=\operatorname{Spec}S,
  \qquad X=\operatorname{Spec}R.
\]
For \(m\ge0\), let \(I_m\subset S\) be the monomial ideal generated by
monomials of weight at least \(m\).  The weighted Rees algebra is
\[
  \mathcal S(V):=\bigoplus_{m\ge0}I_m\ttt^m.
\]

\begin{lemma}\label{lem:normal-invariant-proj}
The graded rings \(\mathcal S(V)\) and \(\mathcal S(V)^G\) are normal.
Consequently, \(\operatorname{Proj}\mathcal S(V)\) and \(\operatorname{Proj}\bigl(\mathcal S(V)^G\bigr)\) are normal.
\end{lemma}

\begin{proof}
List the chosen Jordan coordinates as \(x_1,\ldots,x_N\) and write
\(w_i=\operatorname{wt}(x_i)\).  The monomials of \(\mathcal S(V)\) are
\(x^\alpha\ttt^m\) with \(m\le\sum_i w_i\alpha_i\), so
\[
  \mathcal S(V)=k[\Gamma],
  \qquad
  \Gamma=\Bigl\{(\alpha,m)\in\mathbb N^{N}\times\mathbb N:
  m\le\sum_{i=1}^{N}w_i\alpha_i\Bigr\}.
\]
The semigroup \(\Gamma\) is saturated, so \(\mathcal S(V)\) is normal.  The invariant subring of a normal domain under
a finite group is normal, so \(\mathcal S(V)^G\) is normal.  For any
homogeneous \(h\), the degree-zero part of the
localization \((\mathcal S(V))_h\) or \((\mathcal S(V)^G)_h\) (if \(h\) is invariant) is again normal, and these rings form an
affine open cover of the Proj.
\end{proof}

We then define
\[
  Y:=\operatorname{Proj}\mathcal S(V),
  \qquad
  W:=\operatorname{Proj}\bigl(\mathcal S(V)^{G}\bigr).
\]
Thus \(Y\to V\) is the weighted blowup and \(W\to X\) is the
invariant Proj model, which is normal by
Lemma~\ref{lem:normal-invariant-proj}.  The four spaces fit into the commutative
diagram
\[
\begin{array}{ccc}
Y & \xrightarrow{\ q\ } & W\\
\downarrow g && \downarrow f\\
V & \xrightarrow{\ \pi\ } & X.
\end{array}
\]

Let
\[
  E:=\bigl(\operatorname{Exc}(g)\bigr)_{\mathrm{red}}
\]
be the exceptional prime divisor of the weighted blowup and let
\[
  F:=q(E)\subset W
\]
be its image.

A crepant model
\(h\colon Z\to X\) is a normal birational model satisfying
\[
  K_Z=h^*K_X.
\]
In particular, if \(X\) admits a crepant resolution, then \(D_V=p\) by the canonical non-terminal boundary condition as described
in the introduction.
Unless stated otherwise, in this paper, a crepant resolution or crepant model over \(X\) is
projective over \(X\).  

Finally, \(\chi\) denotes the compactly supported
\(\ell\)-adic Euler characteristic for a prime \(\ell\ne p\).

%% file: pcyclic-sections/all-v2-family.tex
\providecommand{\LL}{\mathbb L}

\section{The construction for the \texorpdfstring{\(V_2^{\oplus n}\)}{V2n}-case}
\label{sec:all-v2-family}

Retain the field \(k\), the prime \(p\), and
\(G=C_p=\langle\sigma\rangle\), and let \(n\ge2\).  Put
\[
  V=V_2^{\oplus n},
  \qquad
  S=k[x_1,\ldots,x_n,y_1,\ldots,y_n],
\]
where
\[
  \sigma(x_i)=x_i,
  \qquad
  \sigma(y_i)=y_i+x_i.
\]
Define
\[
  R=S^G,
  \qquad
  X=\operatorname{Spec}R,
  \qquad
  I=(x_1,\ldots,x_n)\subset S,
  \qquad
  J=(x_1,\ldots,x_n)R,
\]
and recall that
\[
  \mathcal S=S[I\ttt]
  =\bigoplus_{m\ge0}I^m\ttt^m,
\]
where \(\ttt\) is fixed by \(G\).  
In our basic construction, \(Y\) is the ordinary blowup of the fixed locus
\[
  Z=V(x_1,\ldots,x_n)\subset V.
\]
Its exceptional divisor is denoted by \(E\), and
\[
  F:=q(E)\subset W.
\]

\subsection{The affine charts}

On the standard blowup chart \(U_i=D_+(x_i\ttt)\subset Y\), set
\[
  r_i=x_i,
  \qquad
  s_i=y_i,
  \qquad
  c_{ij}=\frac{x_j}{x_i},
  \qquad
  z_{ij}=y_j-c_{ij}y_i
  \quad (j\ne i).
\]
Then
\[
  \sigma(r_i)=r_i,
  \qquad
  \sigma(c_{ij})=c_{ij},
  \qquad
  \sigma(z_{ij})=z_{ij},
  \qquad
  \sigma(s_i)=s_i+r_i.
\]
Thus \(U_i\) is stable under the induced action of \(G\) on the blowup.  Write
\[
  A_i=k[r_i,\{c_{ij}\}_{j\ne i},\{z_{ij}\}_{j\ne i}],
  \qquad
  N_i=s_i^p-r_i^{p-1}s_i.
\]

\begin{lemma}
\label{lem:all-v2-translation-invariants}
For every \(i\),
\[
  A_i[s_i]^G=A_i[N_i].
\]
\end{lemma}

\begin{proof}
The subring \(A_i\) and the element \(N_i\) are \(G\)-invariant, so it remains to show the other inclusion.

The monic relation
\[
  s_i^p-r_i^{p-1}s_i-N_i=0
\]
shows that \(A_i[s_i]\) is generated over \(A_i[N_i]\) by
\[
  1,s_i,\ldots,s_i^{p-1}.
\]
These elements are linearly independent, and form a free \(A_i[N_i]\)-basis.

Let
\[
  h=\sum_{a=0}^{p-1}b_as_i^a,
  \qquad b_a\in A_i[N_i],
\]
be \(G\)-invariant.  If some \(b_a\) with \(a>0\) is nonzero, choose the largest
such \(a\).  In the free basis above, the coefficient of \(s_i^{a-1}\) in
\(\sigma(h)-h\) is \(a r_i b_a\), which is nonzero.  This contradiction leads to \(h=b_0\in A_i[N_i]\).
\end{proof}

\begin{lemma}
\label{lem:all-v2-invariant-localization}
If \(C\) is a \(G\)-algebra and \(u\in C^G\), then
\[
  (C_u)^G=(C^G)_u.
\]
If \(C\) is graded and \(u\) is homogeneous, the equality holds at every degree.
\end{lemma}

\begin{proof}
Represent an invariant element of \(C_u\) by \(a/u^m\).  For every \(g\in G\),
there is an integer \(e\ge0\) such that
\[
  u^{e}(g(a)-a)=0.
\]
If \(e\) is large enough, then \(u^ea\in C^G\) and
\[
  \frac{a}{u^m}=\frac{u^ea}{u^{m+e}}\in(C^G)_u.
\]
The other inclusion is trivial, and the statement for the homogeneous case is from the same argument.
\end{proof}

\begin{proposition}[resolution of singularities]
\label{prop:all-v2-smooth-projective}
The invariant elements \(x_1\ttt,\ldots,x_n\ttt\) define an affine open cover of
\(W\), and
\[
  D_+(x_i\ttt)
  \simeq
  \operatorname{Spec}
  k[r_i,\{c_{ij}\}_{j\ne i},\{z_{ij}\}_{j\ne i},N_i].
\]
Consequently \(W\) is smooth, and
\[
  f:W\longrightarrow X
\]
is projective and birational.
\end{proposition}

\begin{proof}
Every homogeneous element \(u\) of \(\mathcal S\) satisfies its orbit polynomial \(\prod_{g\in G} (T-gu)\)
over \(\mathcal S^G\).  Since \(\mathcal S\) is finitely generated, it is
finite over \(\mathcal S^G\).  If \(v\in\mathcal S\) is homogeneous of
positive degree, the nonleading coefficients of its orbit polynomial lie in
\((\mathcal S^G)_+\), and hence
\[
  v^p\in(\mathcal S^G)_+\mathcal S.
\]
Thus
\[
  \sqrt{(\mathcal S^G)_+\mathcal S}=(\mathcal S)_+.
\]
It follows that
\[
  q:Y=\operatorname{Proj}\mathcal S
  \longrightarrow
  W=\operatorname{Proj}(\mathcal S^G)
\]
is finite and surjective.  The charts \(D_+(x_i\ttt)\) cover \(Y\), and each
\(x_i\ttt\) is invariant.  Their images therefore cover \(W\).

On the \(i\)-th chart, Lemma \ref{lem:all-v2-translation-invariants} and
 \ref{lem:all-v2-invariant-localization} give
\[
\begin{aligned}
  \bigl((\mathcal S^G)_{x_i\ttt}\bigr)_0=\bigl((\mathcal S)_{x_i\ttt}\bigr)_0^G=A_i[s_i]^G
   =A_i[N_i].
\end{aligned}
\]
These charts have polynomial coordinate rings.  Hence \(W\) is smooth.

It remains to show that \(W\to X\) is projective and birational. The graded algebra \(\mathcal S^G\) is finitely generated over its
degree-zero part \(R\).  After passage to a positive Veronese subalgebra, it
is generated in degree one.  Its relative Proj is therefore projective over
\(X=\operatorname{Spec}R\).

The function field of \(U_i\) is
\[
  k(r_i,\{c_{ij}\},\{z_{ij}\},s_i),
\]
with \(G\) acting nontrivially only by \(s_i\mapsto s_i+r_i\), or equivalently \(s_i/r_i \mapsto s_i/r_i+1\). By Artin--Schreier theory, the invariant subfield is
\[
  k(r_i,\{c_{ij}\},\{z_{ij}\},N_i).
\]
It is both \(k(U_i/G)=k(W)\) and \(k(V)^G=k(X)\), so \(f\) is birational.
\end{proof}

\subsection{The construction as a normalized blowup}

Set
\[
  \mathcal B:=R[J\ttt]
  =R[x_1\ttt,\ldots,x_n\ttt].
\]
Then
\[
  \operatorname{Bl}_{J}(X)=\operatorname{Proj}\mathcal B.
\]

\begin{proposition}[normalized blowup]
\label{prop:all-v2-normalized-blowup}
Inside \(Q(R)(\ttt)\),
\[
  \mathcal S^G=\overline{\mathcal B}.
\]
Consequently
\[
  W\simeq
  \operatorname{Norm}\operatorname{Bl}_{J}(X).
\]
\end{proposition}

\begin{proof}
We already know that \(\mathcal S\) and \(\mathcal S^G\) are normal by Lemma~\ref{lem:normal-invariant-proj}.

It is known that \(S\) is finite over \(R=S^G\).
Since
\[
  \mathcal S=S[x_1\ttt,\ldots,x_n\ttt],
\]
it is finite over
\[
  \mathcal B=R[x_1\ttt,\ldots,x_n\ttt].
\]
It follows that \(\mathcal S^G\) is
finite, and hence integral, over \(\mathcal B\), from the fact that \(R\) is noetherian.

We have
\[
  \ttt=\frac{x_1\ttt}{x_1}\in Q(\mathcal B),
\]
and therefore
\[
  Q(\mathcal B)=Q(R)(\ttt).
\]
On the other hand,
\[
\begin{aligned}
  Q(\mathcal S^G)=Q(\mathcal S)^G=Q(S)^G(\ttt)=Q(R)(\ttt).
\end{aligned}
\]
Thus \(\mathcal S^G\) is the integral closure of \(\mathcal B\) in their
common fraction field.

It remains to pass from graded rings to Proj.  Put
\[
  B_i=((\mathcal B)_{x_i\ttt})_0,
  \qquad
  C_i=((\mathcal S^G)_{x_i\ttt})_0.
\]
The \(x_i\ttt\)(\(i=1,\dots,n\)) generate the irrelevant ideal of \(\mathcal B\), and an argument similar to that in the proof of
Proposition~\ref{prop:all-v2-smooth-projective} shows that the corresponding affine charts cover \(\operatorname{Bl}_{J}(X)\).  Since localization preserves integral closure,
\((\mathcal S^G)_{x_i\ttt}\) is the integral closure of
\((\mathcal B)_{x_i\ttt}\).  Therefore every \(C_i\)
is finite and integral over \(B_i\), is normal, and has the same
fractional field \(Q(R)\).

Conversely, any element \(u\in Q(R)\) that is integral over \(B_i\),
is integral over \((\mathcal B)_{x_i\ttt}\). Hence
\(u\in(\mathcal S^G)_{x_i\ttt}\), and therefore lies in \(C_i\).  Thus
\(\operatorname{Spec}C_i\) is the normalization of
\(\operatorname{Spec}B_i\).  These affine normalizations glue to the normalized blowup.
\end{proof}

\subsection{The geometry of the exceptional prime divisor}

\begin{proposition}
\label{prop:all-v2-unique-exceptional}
The divisor \(F\) is the unique \(f\)-exceptional prime divisor.
\end{proposition}

\begin{proof}
The center \(Z\simeq\mathbb A^n\) is smooth and irreducible, and
\[
  E\simeq Z\times\mathbb P^{n-1}
\]
is the unique exceptional prime of \(g:Y\to V\).  Since \(q\) is
finite, \(F=q(E)\) is irreducible of dimension \(2n-1\), hence is a
prime divisor on \(W\).  Its image in \(X\) is \(f(F)=\pi(Z)\).  The finite
map \(\pi:V\to X\) preserves dimensions, so the image of \(Z\) has codimension \(n\ge2\) in \(X\).
Thus \(F\) is \(f\)-exceptional.

Conversely, let \(D\subset W\) be an \(f\)-exceptional prime divisor.  Since
\(q\) is finite and surjective, a prime divisor \(D'\subset Y\) lies
over \(D\).  The commutative diagram gives
\[
  g(D')\subset \pi^{-1}(f(D)).
\]
The right-hand side has codimension at least two in \(V\).  Hence \(D'\) is \(g\)-exceptional, so \(D'=E\) and
\(D=F\).
\end{proof}

Let
\[
  P=\mathbb P^{n-1},
  \qquad
  L=\mathcal O_P(-1),
\]
and let
\[
  0\longrightarrow L
  \longrightarrow\mathcal O_P^{\oplus n}
  \longrightarrow Q
  \longrightarrow0
\]
be the tautological sequence.  The blowup has the vector-bundle description
\[
  Y\simeq
  \operatorname{Tot}_P
  (L\oplus\mathcal O_P^{\oplus n}).
\]
Put
\[
  H=\operatorname{Tot}_P(L\oplus Q).
\]
The natural map \(Y\to H\) sends the \(y\)-coordinates to their class
modulo the tautological line.  On the \(i\)-th chart, the base coordinates are
\(r_i,c_{ij},z_{ij}\), while \(s_i\) is the remaining line coordinate.

Let \(r\) be the universal section of the pullback of \(L\) to \(H\).
The additive polynomial
\[
  \wp_r(t)=t^p-r^{p-1}t
\]
defines a homomorphism from the additive line bundle \(L\) to
\(L^{\otimes p}\).  Its kernel
\[
  \mathcal G_r=
  \ker\bigl(\wp_r:L\longrightarrow L^{\otimes p}\bigr)
\]
is finite flat of rank \(p\), since it is locally defined by a monic
polynomial of degree \(p\).

\begin{proposition}
\label{prop:all-v2-oort-tate}
The morphism \(W\to H\) is an affine torsor under the additive line bundle
\(L^{\otimes p}\), and \(Y\to W\) is a torsor under the pullback of
\(\mathcal G_r\) to \(W\).

Let
\[
  H_0=V(r)\subset H.
\]
Then
\[
  H_0\simeq\operatorname{Tot}_P(Q),
  \qquad
  F=W\times_H H_0
\]
scheme-theoretically.  Along \(F\), the group scheme \(\mathcal G_r\)
specializes to
\[
  \ker\bigl(
    \operatorname{Fr}:L|_{F}
    \longrightarrow L^{\otimes p}|_{F}
  \bigr),
\]
which is locally \(\alpha_p\).  The induced map
\[
  E\longrightarrow F
\]
is finite flat of degree \(p\), radicial, and a universal homeomorphism.
\end{proposition}

\begin{proof}
On the \(i\)-th chart, the quotient map \(Y\to W\) is given by
\[
  A_i[N_i]\longrightarrow A_i[s_i],
  \qquad
  N_i\longmapsto s_i^p-r_i^{p-1}s_i.
\]
The local group scheme
\[
  \operatorname{Spec}
  A_i[t_i]/(t_i^p-r_i^{p-1}t_i)
\]
acts by \(s_i\mapsto s_i+t_i\).  The torsor identity is
\[
\begin{aligned}
  s_i'{}^p-r_i^{p-1}s_i'
  &=s_i^p-r_i^{p-1}s_i\\
  &\Longleftrightarrow
  (s_i'-s_i)^p-r_i^{p-1}(s_i'-s_i)=0.
\end{aligned}
\]

On the overlap of the \(i\)-th chart and the \(j\)-th chart, write \(c=x_j/x_i\).  Then
\[
  r_j=cr_i,
  \qquad
  s_j=cs_i+z_{ij},
\]
and direct substitution gives
\[
  N_j
  =c^pN_i+z_{ij}^p-c^{p-1}r_i^{p-1}z_{ij}.
\]
The linear part \(N_i\mapsto c^pN_i\) is the transition function of
\(L^{\otimes p}\), while the remaining term depends only on the
\(H\)-coordinates.  Hence the \(N_i\) glue to an affine torsor under
\(L^{\otimes p}\).  The parameters \(t_j=ct_i\) glue the local kernels to
\(\mathcal G_r\), and the local torsor identities glue as well.

The equation \(r=0\) corresponds to
\(\operatorname{Tot}_P(Q)\subset H\).  On the \(i\)-th quotient
chart, its pullback is \(r_i=0\), which cuts out \(F\).  After setting
\(r_i=0\), the group scheme and cover become
\[
  t_i^p=0,
  \qquad
  N_i=s_i^p.
\]
Thus the group scheme is locally \(\alpha_p\), and the cover is finite free
of rank \(p\).  Over an algebraic closure of any residue field, the equation
\(T^p-N_i\) has a unique geometric root.  The map \(E\to F\) is therefore radicial and
hence a universal homeomorphism.
\end{proof}

\subsection{Classes and Euler characteristics}

Let \(K_0(\operatorname{Var}_k)\) be the Grothendieck ring of \(k\)-varieties,
and define
\[
  \LL=[\mathbb A^1]\in K_0(\operatorname{Var}_k).
\]

\begin{corollary}
\label{cor:all-v2-grothendieck-classes}
In \(K_0(\operatorname{Var}_k)\),
\[
  [W]=\LL^{n+1}[\mathbb P^{n-1}],
  \qquad
  [F]=\LL^n[\mathbb P^{n-1}].
\]
\end{corollary}

\begin{proof}
An affine torsor under a line bundle is Zariski locally trivial with fiber
\(\mathbb A^1\).  Proposition \ref{prop:all-v2-oort-tate} gives
\[
  [W]=\LL[H],
  \qquad
  [F]=\LL[H_0].
\]
The space \(H\) is the total space of the rank-\(n\) vector bundle
\(L\oplus Q\) on \(\mathbb P^{n-1}\), while \(H_0\) is the total
space of the vector bundle \(Q\) of rank \((n-1)\).  The formulas follow.
\end{proof}

Fix a prime \(\ell\ne p\).  For a \(k\)-variety \(Z\), let
\[
  \chi(Z)
  :=\sum_i(-1)^i
  \dim_{\mathbb Q_\ell}
  H_{c,\mathrm{\acute et}}^i(Z,\overline{\mathbb{Q}_\ell})
\]
be its compactly supported \(\ell\)-adic Euler characteristic.

\begin{corollary}
\label{cor:all-v2-euler}
One has
\[
  \chi(W)=\chi(F)=n.
\]
In particular, when \(n=p\), the Euler characteristic is
\(p=\#\operatorname{Ind}_k(C_p)\),
where \(\operatorname{Ind}_k(C_p)\) denotes the set of isomorphism classes of
indecomposable finite-dimensional \(kC_p\)-modules.
\end{corollary}

\begin{proof}
The compactly supported \(\ell\)-adic Euler characteristic factors through
\(K_0(\operatorname{Var}_k)\), sends \(\LL\) to \(1\), and sends
\([\mathbb P^{n-1}]\) to \(n\).  The first assertion follows from Corollary
\ref{cor:all-v2-grothendieck-classes}.  When \(n=p\), combining with the fact that \(W\to X\) is crepant, the assertion meets Yasuda's \(p\)-cyclic McKay correspondence in \cite{Yasuda2014pCyclic} and Fan's conjecture in \cite{Fan2024EulerCrepant}.
\end{proof}

\subsection{The discrepancy of \(F\)}

\begin{proposition}
\label{prop:all-v2-canonical-divisor}
\[
  K_{W}=f^*K_{X}+(n-p)F.
\]
In particular, \(W\to X\) is a crepant morphism if and only if \(n=p\).
\end{proposition}

\begin{proof}
The action of \(C_p\) on \(V\) has determinant one, and \(\pi\) is \'etale in codimension one.  Therefore, we can choose \(K_V=\pi^*K_X=0\).  Write
\[
  \Omega=
  dx_1\wedge\cdots\wedge dx_n
  \wedge dy_1\wedge\cdots\wedge dy_n.
\]
Let \(\eta\) be the quotient top form on \(X\), normalized by
\(\pi^*\eta=\Omega\).
On \(U_i\subset Y\), define
\[
  \omega_i=
  dr_i\wedge\bigwedge_{j\ne i}dc_{ij}
  \wedge ds_i\wedge\bigwedge_{j\ne i}dz_{ij}.
\]
The coordinate change
\[
  x_i=r_i,\quad
  x_j=c_{ij}r_i,\quad
  y_i=s_i,\quad
  y_j=z_{ij}+c_{ij}s_i
\]
gives
\[
  g^*\Omega=\pm r_i^{n-1}\omega_i.
\]
On the corresponding quotient chart of \(W\), put
\[
  \Theta_i=
  dr_i\wedge\bigwedge_{j\ne i}dc_{ij}
  \wedge dN_i\wedge\bigwedge_{j\ne i}dz_{ij}.
\]
In characteristic \(p\),
\[
  dN_i
  =d(s_i^p-r_i^{p-1}s_i)
  =r_i^{p-2}s_i\,dr_i-r_i^{p-1}ds_i.
\]
The \(dr_i\)-term vanishes in the wedge product, so
\[
  q^*\Theta_i=\pm r_i^{p-1}\omega_i.
\]
Since
\[
  q^*f^*\eta=g^*\pi^*\eta=g^*\Omega,
\]
comparison gives
\[
  f^*\eta=\pm r_i^{n-p}\Theta_i.
\]
The equation \(r_i=0\) corresponds to \(F\) on the quotient chart and \(E\) on
its cover.  Thus, the ramification index normal to the exceptional divisor is
one; the inseparable behavior in Proposition \ref{prop:all-v2-oort-tate} is
tangent to that divisor.  The quotient charts cover \(W\), and Proposition
\ref{prop:all-v2-unique-exceptional} shows that \(F\) is the only
exceptional prime.  Therefore its discrepancy in \(W\to X\) is \(n-p\).
\end{proof}

\begin{proof}[Proof of Theorem~\ref{thm:all-v2}]
The statement of the theorem has been discussed through this chapter.
When \(n=p\), \(f\) is a projective crepant
resolution and \(\chi(W)=p\) by Proposition~\ref{prop:all-v2-smooth-projective}, Corollary~\ref{cor:all-v2-euler} and Proposition~\ref{prop:all-v2-canonical-divisor}.
\end{proof}

\begin{corollary}
\label{cor:all-v2-boundary-endpoint}
For every integer \(a\ge0\),
\[
  (V_2^{\oplus p}\oplus V_1^{\oplus a})/C_p
  \simeq
  (V_2^{\oplus p}/C_p)\times\mathbb A^a
\]
admits a projective crepant resolution.
\end{corollary}

%% file: pcyclic-sections/higher-block-obstruction.tex
\section{Higher Jordan blocks}\label{sec:higher-blocks}

Let us go back to the general setting.

\subsection{The discrepancy of the exceptional divisor}

\begin{lemma}\label{lem:hb-tail-dvr}
Let \(B=\mathcal O_{Y,E}\) be the generic exceptional DVR of the weighted
blowup \(Y\to V\).  Choose a nontrivial Jordan block \(V_d\) with $d\geq 2$ and put
\[
  r=x_{d-1},\qquad s=x_d.
\]
Then
\[
  I_G(B):=(\sigma(b)-b:b\in B)=(r).
\]
For the finite extension of generic exceptional DVRs
\[
  B^G=\mathcal O_{W,F}\subset B,
\]
the ramification index and different exponent are
\[
  e(E/F)=1,\qquad d(B/B^G)=p-1.
\]
\end{lemma}

\begin{proof}
The Jordan weights of \(r\) and \(s\) are one and zero respectively, and
\[
  \sigma(s)=s+r.
\]
On the chart \(D_+(r\ttt)\), a coordinate \(z\) of positive weight \(v\)
has the form \(z=Z_zr^v\).  The Jordan action gives
\[
  \frac{\sigma(r)}r=1+r\cdot(\text{regular element}),
  \qquad \sigma(Z_z)-Z_z\in(r).
\]
The same statement holds for the coordinates of weight zero. Hence
\((\sigma-1)B\subset(r)\), while \(r=\sigma(s)-s\) gives the reverse
inclusion.

For the extension \(B/B^G\), if the residue
of \(s\) belongs to the residue field of \(B^G\), an invariant lift
\(s_0\in B^G\) would satisfy \(\ord_E(s_0-s)>0\).  Then
\[
  \ord_E\bigl(\sigma(s_0-s)-(s_0-s)\bigr)\ge2,
\]
contrary to the computation \(\sigma(s_0-s)-(s_0-s)=-r\).  Since the extension has degree \(p\), the residue degree has to be
\(p\), and its ramification
index is one.

For \(a\in\Fp^\times\),
\[
  \sigma^a(s)-s=ar+O(r^2),
  \qquad \ord_E(\sigma^a(s)-s)=1.
\]
Put \(A=B^G\).  The residue of \(s\) has degree \(p\) over the residue
field of \(A\), so \(1,\bar s,\ldots,\bar s^{p-1}\) is a basis of the
residue field of \(B\).  Nakayama's lemma gives \(B=A[s]\).  The orbit
polynomial of \(s\) is its minimal polynomial over \(Q(A)\), and the
monogenic different formula gives
\[
  d(B/A)
  =\ord_E\!\left(\prod_{a\in\Fp^\times}
    (s-\sigma^a(s))\right)
  =p-1.
\]
\end{proof}

\begin{proposition}\label{prop:hb-coefficient}
The coefficient of \(F\) in \(K_W-f^*K_X\) is
\[
  a(F;X)=D_V-p.
\]
In particular, it is zero when \(D_V=p\).
\end{proposition}

\begin{proof}
The Jacobian order of the weighted blowup \(Y\to V\) along \(E\) is
\(D_V-1\).  The computation of the coefficient of \(E\) in \(K_Y-g^*\pi^*K_X\) gives 
\[
  D_V-1=e(E/F)a(F;X)+d(B/B^G).
\]
By Lemma~\ref{lem:hb-tail-dvr}, \(a(F;X)=D_V-p\).
\end{proof}

\subsection{Root charts of the natural model}

Let \(u\) be a coordinate of positive weight \(w\), and set
\[
  H_u=\prod_{i\in\Fp}\sigma^i(u\ttt^w)\in\calS(V)^G.
\]

\begin{proposition}[invariant root charts]\label{lem:hb-root-charts-01}
The opens \(D_+(H_u)\), with \(u\) ranging over the
coordinates of positive weights, cover \(\Proj(\calS(V)^G)\).  
\end{proposition}

\begin{proof}
The graded algebra \(\calS(V)\) is generated over \(S\) by
\[
  z\ttt^a,\qquad 1\le a\le\operatorname{wt}(z).
\]
If a homogeneous prime $I$ does not contain \(z\ttt^a\), then
\[
  (z\ttt^a)^{\operatorname{wt}(z)}
  =z^{\operatorname{wt}(z)-a}
   (z\ttt^{\operatorname{wt}(z)})^a
\]
shows that it cannot contain \(z\ttt^{\operatorname{wt}(z)}\).  Choose \(u\ttt^w\) of maximal weight among the generators such that \(u\ttt^w \notin I\).  For any \(v>w\) and generator \(z\ttt^v\), the identity
\[
  (z\ttt^w)^v=z^{v-w}(z\ttt^v)^w
\]
shows that \(z\ttt^w\) also belongs to the prime.
Therefore,
\[
  \sigma^i(u\ttt^w)\equiv u\ttt^w
\]
modulo the prime, for every \(i\).  Hence, the prime does not contain \(H_u\).  Lying over for
\(\calS(V)^G\subset\calS(V)\) proves that the invariant opens cover the invariant \(\Proj\).

\end{proof}

\begin{lemma}\label{lem:hb-root-charts-02}
Put
\[
  A_u:=\left((\calS(V)^G)_{H_u}\right)_0,
  \qquad
  C_u:=\left(\calS(V)_{H_u}\right)_0.
\]
Let \(\mathcal X_u\) be the set of chosen Jordan coordinates other
than \(u\).  Then there is a regular ring
\[
  P_u=
  k[\rho,\{Z_z\}_{z\in\mathcal X_u},\Delta_u^{-1}],
\]
where
\[
  u=\rho^w,\qquad
  z=Z_z\rho^{\operatorname{wt}(z)}\quad(z\in\mathcal X_u),\qquad
  \Delta_u=\prod_{i\in\Fp}\frac{\sigma^i(u)}u,
\]
such that
\[
  P_u^{\mu_w}=C_u,\qquad A_u=C_u^G.
\]
\end{lemma}

\begin{proof}
Because \(H_u\) is invariant, localization commutes with invariants, and
\[
  A_u=C_u^G.
\]
The equality \(H_u=(u\ttt^w)^p\Delta_u\) identifies \(C_u\) with the
degree-zero ring of \(D_+(u\ttt^w)\) since \(\Delta_u\) is invertible.  On the
root cover, \(\mu_w\) acts by
\[
  \rho\longmapsto\zeta\rho,\qquad
  Z_z\longmapsto\zeta^{-\operatorname{wt}(z)}Z_z.
\]
A monomial \(\rho^bZ^\alpha\) is fixed precisely when
\(b-\sum_z\operatorname{wt}(z)\alpha_z\) is divisible by \(w\).  Such
monomials are exactly the pullbacks of degree-zero chart monomials.  Hence
\(P_u^{\mu_w}=C_u\).
\end{proof}

\begin{lemma}\label{lem:hb-root-charts-03}
There is a finite \'etale regular \(P_u\)-algebra \(B_u\) and a
finite group \(\Gamma_u\) with \(A_u=B_u^{\Gamma_u}\).  After strict
henselization at a chosen geometric point and completion, every local factor is an inertia quotient
\[
  \widehat{A_u^{\mathrm{sh}}}
  \simeq \widehat{B_{u,q}^{\mathrm{sh}}}^{\,P_{u,q}\rtimes T_{u,q}},
  \qquad
  P_{u,q}\in\{1,C_p\},\quad T_{u,q}\le\mu_w.
\]
If \(P_{u,q}=C_p\), the wild
quotient \(\widehat{B_{u,q}^{\mathrm{sh}}}^{\,P_{u,q}}\) is regular.
\end{lemma}

\begin{proof}
For \(i\in\Fp^\times\), put \(F_i=\sigma^i(u)/u\) and adjoin
\(Q_i\) with \(Q_i^w=F_i\). Take
\[
  B_u=P_u[Q_i\mid i\in\Fp^\times]/(Q_i^w-F_i).
\]
Here \(F_i\) are units on \(D_+(H_u)\), and \(w<p\), so \(B_u\) is finite
\'etale and regular over \(P_u\).  Let
\(D_Q\simeq(\mu_w)^{p-1}\) act on each \(Q_i\) independently.  Then
\(B_u^{D_Q}=P_u\).  The root group \(\mu_w\) acts on \(P_u\) as in the proof of Lemma~\ref{lem:hb-root-charts-02} and
fixes every \(Q_i\).  Put \(Q_0=Q_p=1\) and \(\rho_i=\rho Q_i\), so that
\(\rho_i^w=\sigma^i(u)\) and \(\rho_p=\rho_0\).  The \(C_p\)-action lifts by
\[
  \sigma(\rho_i)=\rho_{i+1},\qquad
  \sigma(\rho)=\rho Q_1,\qquad
  \sigma(Q_i)=Q_{i+1}/Q_1,
\]
and, for \(z\in\mathcal X_u\), by
\[
  \sigma(Z_z)=
  \frac{\sigma(z)}{(\rho Q_1)^{\operatorname{wt}(z)}}.
\]
These formulas cyclically permute the roots \(\rho_i\)
and recover the original action on every coordinate \(z\); hence the lift has
order \(p\).
It normalizes \(D_Q\times\mu_w\), and for
\[
  \Gamma_u=(D_Q\times\mu_w)\rtimes C_p
\]
one obtains
\[
  B_u^{\Gamma_u}=(P_u^{\mu_w})^G=C_u^G=A_u.
\]
In particular, \(A_u\) is normal, in agreement with
Lemma~\ref{lem:normal-invariant-proj}.

After strict henselization of \(A_u\) at a geometric point, the corresponding \(B_u^{\mathrm{sh}}:=B_u\otimes_A A^{\mathrm{sh}}\) splits into local factors.  The factors
form one orbit under \(\Gamma_u\), because otherwise the sum of the idempotents over one orbit
would give a nontrivial idempotent in the local invariant ring \((B_u^{\mathrm{sh}})^{\Gamma_u}=A_u^{\mathrm{sh}}\).  The completed
invariant ring \(\widehat{A_u^{\mathrm{sh}}}\) is therefore the invariant ring of one completed factor \(\widehat{B_{u,q}^{\mathrm{sh}}}\) under
its inertia group.  A nontrivial element of \(D_Q\) changes the residue of
some \(Q_i\), so \(D_Q\) does not contribute to the inertia group.  Hence the tame part of the inertia group \(T_{u,q}\) is
contained in \(\mu_w\), and the wild part \(P_{u,q}\) is either trivial or \(C_p\).

Suppose that the wild part is nontrivial.  Coprime cohomology gives
\(H^1(C_p,D_Q\times\mu_w)=0\).  A twisted lift of the wild generator \(\sigma\) is
therefore conjugate to the standard lift, which acts on the residues by
\[
  \bar Q_{i+1}=\bar Q_1\bar Q_i,\qquad
  \bar Q_0=\bar Q_p=1.
\]
Thus \(\bar Q_i=\bar Q_1^i\) and \(\bar Q_1^p=1\).  In characteristic \(p\),
this forces \(\bar Q_i=1\) for every \(i\).  In normalized coordinates,
\[
  Q_i^w=\frac{\sigma^i(u)}u=1+\rho A_i.
\]
Since
\[
  Q_i^w-1=(Q_i-1)(1+Q_i+\cdots+Q_i^{w-1})
\]
and the second factor has residue \(w\in k^\times\), one has
\(Q_i-1\in(\rho)\).  The formulas of the action of \(\sigma\) now give
\[
  (\sigma-1)\widehat{B_{u,q}^{\mathrm{sh}}}\subset(\rho).
\]

We want to show the other inclusion. Below we let \(u=x_j\), write \(S_0=1\) and
\[
  x_{j+\ell}=S_\ell\rho^{w-\ell},
  \qquad 0\le\ell\le w.
\]
Then
\[
  \sigma(S_\ell)-S_\ell=\rho\varepsilon_\ell.
\]
Write \(Q_1=1+\rho A\).  Since \(\sigma(x_{j+\ell})=x_{j+\ell}+x_{j+\ell-1}\)
and \(\sigma(\rho)=\rho Q_1\), we have
\[
  \sigma(S_\ell)=\frac{S_\ell+\rho S_{\ell-1}}{Q_1^{\,w-\ell}},
  \qquad\text{hence}\qquad
  \bar\varepsilon_\ell=\bar S_{\ell-1}-(w-\ell)\bar A\,\bar S_\ell
\]
for \(1\le\ell\le w\).  If no \(\varepsilon_\ell\) is a unit, then every
\(\bar\varepsilon_\ell\) vanishes.  The case \(\ell=w\) gives
\(\bar S_{w-1}=0\), and descending induction gives
\(\bar S_{w-1}=\cdots=\bar S_0=0\), contrary to \(S_0=1\).  Thus some
\(\varepsilon_\ell\) is a unit, so \(\rho=\varepsilon_\ell^{-1}(\sigma-1)S_\ell\)
lies in the augmentation ideal \(((\sigma-1)\widehat{B_{u,q}^{\mathrm{sh}}})\), which is therefore \((\rho)\).
By Kir\'aly--L\"utkebohmert's theorem, a prime-order action on a regular local
ring with principal augmentation ideal has a regular invariant
ring~\cite{KiralyLutkebohmert2010}.
\end{proof}

From the lemma above, for each affine chart of \(W\), its singularity is governed by the tame part \(T_{u,q}\). We consider the action of \(T_{u,q}\) on the quotient by \(P_{u,q}\) in the following lemmas.

\begin{lemma}
\label{lem:hb-character-survival}
Use the notation of Lemma~\ref{lem:hb-root-charts-03}, write
\(P=P_{u,q}\) and \(T=T_{u,q}\), and suppose that \(P=C_p\).  For the
maximal ideal filtration on the completed local factor, there is a
\(T\)-eigen parameter system
\[
  \rho,\quad \Theta,\quad Z_1,\ldots,Z_e
\]
on the associated graded ring, such that the graded action \(\sigma_0\) gives
\[
  \sigma_0(\rho)=\rho,\qquad
  \sigma_0(\Theta)=\Theta+\rho,\qquad
  \sigma_0(Z_j)=Z_j.
\]
The parameters \(\rho\) and \(\Theta\) have the same \(T\)-eigenvalues.  The
eigenvalues of the \(Z_j\)'s are given in the same way as the action of \(\mu_w\) described in the proof of Lemma~\ref{lem:hb-root-charts-02}.
There are \(P\)-invariant, \(T\)-eigen elements
\[
  \widetilde\rho,\qquad
  \widetilde N_\Theta,\qquad
  \widetilde Z_1,\ldots,\widetilde Z_e
\]
whose initial forms are
\[
  \rho,\qquad
  N_\Theta=\Theta^p-\rho^{p-1}\Theta,\qquad
  Z_1,\ldots,Z_e,
\]
which forms a regular system of parameters of the \(P\)-invariant local ring.  If an element of \(T\) has eigenvalue \(\zeta^a\) on the initial form
\(\rho\), then its eigenvalues on these lifted classes are \(\zeta^a\), \(\zeta^{pa}\), and
\(\zeta^{-\operatorname{wt}(z_j)a}\) for those coming from the weighted coordinates.
\end{lemma}

\begin{proof}
Let \(\sigma_0\) be the graded action on \(\widehat{B_{u,q}^{\mathrm{sh}}}\) such that for each homogeneous element \(b\) of degree \(d\), \(\sigma_0(b)\) is the degree-\(d\) part of \(\sigma(b)\). The proof of Lemma~\ref{lem:hb-root-charts-03} gives
\(\sigma(S_\ell)-S_\ell=\rho\varepsilon_\ell\) with some unit
\(\varepsilon_\ell\).  After centering \(S_\ell\) at the chosen
geometric point and rescaling its initial form if necessary, we obtain the form \(\Theta\), such that
the operator \(\sigma_0-1\) has the image spanned by \(\rho\), sends
\(\Theta\) to \(\rho\), and kills \(\rho\).  By adding
\(T\)-eigen parameters \(Z_1,\ldots,Z_e\) that are normalized by \(\rho^{-\operatorname{wt}(z_j)}\) from the \(\sigma_0\)-invariant coordinates \(z_1,\ldots,z_e\), we have a basis of the kernel.  Moreover, \(\Theta\) has the same \(T\)-eigenvalues as
\(\rho\) since \(\sigma_0(\Theta)=\Theta+\rho\).

For the graded \(\sigma_0\)-invariant ring, it is regular with a system of parameters
\[
  \rho,\quad N_\Theta,\quad Z_1,\ldots,Z_e.
\]
By defining the trace map \(\displaystyle\Tr:=\sum_{i=0}^{p-1}\sigma_0^i\) with the properties
\[
\Tr(\Theta^i)=0 (0\leq i \leq p-2), 
\qquad
\Tr(\Theta^{p-1})=-\rho^{p-1},
\]
together with the basis \(1,\Theta,\ldots,\Theta^{p-1}\) over the graded invariant ring
\(A_0:=k[\rho,N_\Theta,Z_1,\ldots,Z_e]\), we have
\[
  \operatorname{im}(\sigma_0-1)=\rho\cdot\bigoplus_{i=0}^{p-2}A_0\Theta^i=\ker(\Tr)\cap\rho\operatorname{gr}\widehat{B_{u,q}^{\mathrm{sh}}}
\]
in the associated graded ring.  Start with arbitrary lifts of the
displayed initial forms in \(\widehat{B_{u,q}^{\mathrm{sh}}}\).  At each filtration degree of the images of \(\sigma-1\) of the lifts, the graded images
are trace-zero because \(\Tr(\sigma-1)=0\), and divisible by \(\rho\); the identity above then removes the difference by changing the lifts in the given degree.  Because of the completeness, such operations by degree converge to
\(P\)-invariant lifts.

Write \(\widehat B=\widehat{B_{u,q}^{\mathrm{sh}}}\) and \(A=\widehat B^{P}\).
The standard lift of \(\sigma\) commutes with \(\mu_w\), since \(\mu_w\) fixes
every coordinate \(z\) and every \(Q_i\).  After the conjugation in the proof of
Lemma~\ref{lem:hb-root-charts-03}, the tame group \(T\) therefore commutes with
\(P\), so \(T\) acts on \(A\) and \(\sigma-1\) preserves each \(T\)-eigenspace.
Since \(|T|\) is invertible in \(k\), projecting an invariant lift to the
\(T\)-eigenspace of its initial form keeps it \(P\)-invariant and does not
change its initial form.  Hence the lifts can be chosen to be \(T\)-eigen, with
the same eigenvalues as their initial forms.

It remains to see that the lifts generate the maximal ideal of \(A\).  Let
\(I\subset A\) be the ideal they generate.  In the associated graded ring,
the initial forms \(\rho,N_\Theta,Z_1,\ldots,Z_e\) generate an ideal with
quotient \(k[\Theta]/(\Theta^p)\), so \(\widehat B/I\widehat B\) has length at
most \(p\).  On the other hand, \(\widehat B\) is finite over the regular local
ring \(A\) and is Cohen--Macaulay, hence free by the Auslander--Buchsbaum
formula, of rank \(p\).  Since \(P\) is an inertia group, \(A\) and \(\widehat B\)
have the same residue field, and therefore
\[
  \operatorname{length}_{\widehat B}(\widehat B/I\widehat B)
  =p\cdot\operatorname{length}_A(A/I).
\]
Thus \(\operatorname{length}_A(A/I)=1\), that is, \(I\) is the maximal ideal
of \(A\).  So the lifts form a regular system of parameters of \(A\).
\end{proof}

\subsection{The terminal obstruction}
In this section, we prove that there is a root chart with a terminal \(\mu_2\)-quotient singularity when \(D_V=p\ge5\), and \(V\) contains a Jordan block
of size at least three.

\begin{lemma}\label{lem:hb-age}
Assume that \(D_V=p\ge5\), and \(V\) contains a Jordan block
of size at least three. On every completed local factor in Lemma~\ref{lem:hb-root-charts-03}, each
nontrivial element \(t\) of the residual tame group acts on the \(P\)-invariant regular
local ring with age greater than one, where the age of \(t\) with eigenvalues \(\zeta^{a_i}\)\((0\leq a_i<w)\) is defined as \(\displaystyle\frac{1}{w}\sum a_i\), according to \cite{ItoReid1996}.
\end{lemma}

\begin{proof}
Let \(t\) have eigenvalue \(\zeta^a\) \((0< a<w)\) on the root coordinate and on the lifted class with initial
form \(\rho\), and let \(m=w/\gcd(w,a)\) be the order of \(\zeta^a\).  Since \(T\le\mu_w\), one has \(m\mid w<p\).
If the wild subgroup is nontrivial, the
\(P\)-invariant regular local ring has an invariant class with initial form
\[
  N_\Theta=\Theta^p-\rho^{p-1}\Theta,
\]
on which \(t\) has eigenvalue \(\zeta^{pa}\) by
Lemma~\ref{lem:hb-character-survival}.  This eigenvalue is nontrivial, since
\(1<m<p\) and hence \(m\nmid p\).

Assume first that \(m\ge3\).  Then \(w\ge3\), and the block containing
\(u\) has a coordinate of weight \(1\).  The corresponding normalized coordinate having \(t\)-eigenvalue \(\zeta^{-a}\), together
with \(\zeta^a\) from \(\rho\), contributes one to the age. If the wild subgroup is trivial, the parameter from the coordinate of weight
\(w-1\) with eigenvalue \(\zeta^{-(w-1)a}\) contributes \(\frac{a}{w}\) to the age additionally.  If the wild subgroup is nontrivial, because
\(\zeta^{-a}\ne\zeta^{a}\) when \(m\ge3\), the direction from the coordinate of weight \(1\) is different from \(N_\Theta\). Then \(N_\Theta\) contributes a nonzero part to the age.  Thus the age is
greater than one.

It remains to consider \(m=2\).  Let \(O(V)\) be the number of odd
Jordan weights.  Since the sum of Jordan weights is \(D_V=p\), the integer \(O(V)\) is odd.
If \(O(V)=1\), the only possible case is when \(D_V=3\) with the only non-trivial Jordan block of size \(3\). Since \(p\geq 5\) by assumption, \(O(V)\ge3\).
The root coordinate \(\rho\) and the odd normalized directions give at least four
directions, such that each direction contributes \(\frac{1}{2}\) to the age.  (If one of these directions is replaced by
\(N_\Theta\), the contribution by it is again \(\frac{1}{2}\) because \(p\) is odd.)
The age is therefore at least \(2\).
\end{proof}

\begin{remark}\label{rem:V3-chart-age}
The tame action can have the age \(1\) only when \(p=3\) and \(V=V_3\) after removing the trivial summands. In this case, by computation, the square-root chart of \(W\) carries \(A_1\)-singularities, which can be resolved by a crepant blowup; the other invariant root chart of \(W\) is smooth. Since \(W\to X\) is crepant (we will prove this property in the next section), we then obtain a crepant resolution \(\widetilde{W}\to W\to X\) by composing two blowups. A crepant resolution of \(V_3/C_3\) is also constructed in \cite[Example~6.23]{Yasuda2014pCyclic} by the composition of two blowups, and we can check that the two blowups in Yasuda's construction are isomorphic to \(W\to X\) and \(\widetilde{W}\to W\) respectively.
\end{remark}

The preceding lemmas show that under the assumption in this section, \(W\) is terminal.  We then exhibit a
singular point on a weight-two root chart.

\begin{proposition}
\label{prop:hb-terminal-singular}
Assume that \(D_V=p\ge5\), and \(V\) contains a Jordan block \(V_d\)
of size \(d\ge3\). Then the model \(W\) is terminal. Moreover, let
\(H_u\) be the orbit norm associated with its weight-two coordinate
\[
  u=x_{d-2}
\]
as in Proposition~\ref{lem:hb-root-charts-01}.  There is a geometric point
\(\bar w\in D_+(H_u)\) and integers \(r_0\ge0\), \(M\ge4\) such that
\[
  \widehat{\mathcal O_{W,\bar w}^{\mathrm{sh}}}
  \simeq
  k[[v_1,\ldots,v_{r_0}]]\widehat\otimes
  k[[z_1,\ldots,z_M]]^{\mu_2},
\]
where \(\mu_2\) fixes the \(v_i\)'s and sends every \(z_j\) to \(-z_j\).
Consequently, \(W\) is singular.
\end{proposition}

\begin{proof}
Lemma~\ref{lem:hb-root-charts-03} reduces each completed geometric local ring to
a tame quotient of a regular local ring.  Lemma~\ref{lem:hb-age} shows that every such quotient is either regular (if the tame group \(T\) is trivial) or terminal
\cite{Reid1987YoungPerson}.  Hence \(W\) is
terminal.

For singularity, take the square-root chart with \(u=\rho^2\). For the complete local factor, choose the geometric point with \(Q_i=1\), \(\rho=0\), and all
root-normalized \(\mu_2\)-eigen parameters equal to zero.
The formulas in Lemma~\ref{lem:hb-root-charts-03} show that this point is fixed by
the standard lift of \(C_p\) and by the root involution.  Its inertia therefore
contains \(C_p\rtimes\mu_2\).  Write the adjacent coordinate of weight \(1\) as
\(x_{d-1}=\rho S\).  On the associated graded ring, the graded action gives
\[
  \sigma_0(\rho)=\rho,
  \qquad
  \sigma_0(S)=S+\rho.
\]
And \(\mu_2\) acts by \(\rho \mapsto -\rho\), \(S\mapsto -S\).
By Lemma~\ref{lem:hb-character-survival}, the \(C_p\)-invariant local subring has two independent classes with non-trivial \(\mu_2\)-action, whose initial forms are
\[
  \rho,\qquad N_S=S^p-\rho^{p-1}S.
\]
We count two further sign directions.  If some block has size at least four,
then the number \(O(V)\) of odd positive Jordan weights is at least two and,
because \(O(V)\equiv D_V\equiv1\pmod2\), at least three.  One odd direction
is the selected direction from \(x_{d-1}\), which is already represented by \(N_S\); at
least two other odd normalized coordinates remain.  Together with \(\rho\)
and \(N_S\), they give at least four sign directions.

It remains to consider the case in which every nontrivial block has size two
or three.  Write
\[
  V=V_3^{\oplus a}\oplus V_2^{\oplus b}\oplus V_1^{\oplus c},
  \qquad a\ge1.
\]
The condition \(D_V=p\) is \(3a+b=p\).  Since \(p\ge5\) is prime, \(b>0\). In particular, \(b>1\) if \(a=1\). Therefore, \(a+b\geq 3\).
On the selected root chart, the sign directions are \(\rho,N_S\), the
weight-one coefficients of the other \(a-1\) copies of \(V_3\), and the
\(b\) weight-one coordinates from the \(V_2\)-summands.  Their number is
\[
  M=2+(a-1)+b=a+b+1\ge4.
\]

Because \(2\) is invertible in \(k\), the involution on the complete regular local \(C_p\)-quotient can be linearized.
The completed local ring is therefore the displayed sign quotient.  Its
nontrivial element has age \(M/2\ge2\), so the local ring is terminal; it is
singular because the sign action is nontrivial.
\end{proof}

\subsection{Exceptional divisor and the valuation model}

\begin{lemma}\label{lem:hb-support}
The morphism \(f:W\to X\) is projective and birational.  It has exactly one
exceptional prime divisor, namely \(F=q(E)\).
\end{lemma}

\begin{proof}
We follow the proofs of Proposition~\ref{prop:all-v2-smooth-projective} and
Proposition~\ref{prop:all-v2-unique-exceptional}.

Every homogeneous element of \(\calS(V)\) satisfies its orbit polynomial over
\(\calS(V)^G\), so \(\calS(V)\) is finite over \(\calS(V)^G\).  The invariant
algebra is finitely generated over its degree-zero part \(R\), and after
passage to a positive Veronese subalgebra it is generated in degree one.
Hence \(W\) is projective over \(X\).  As in the proof of
Proposition~\ref{prop:all-v2-smooth-projective}, the irrelevant ideal of
\(\calS(V)\) is the radical of the ideal generated by \(\calS(V)^G_+\), so the
finite extension induces a finite surjective morphism \(q:Y\to W\).  The
function field of \(W\) is the degree-zero part of
\(Q(\calS(V)^G)=Q(\calS(V))^G\), which is \(k(V)^G=k(X)\).  Hence \(f\) is
birational.

The center of the weighted blowup is the linear subspace
\(Z\subset V\) cut out by the coordinates of positive weight.  It is
irreducible, and its codimension is the number of such coordinates, which is
at least two since \(D_V=p\ge2\).  The exceptional locus of \(g\) is the
Cartier divisor cut out by \(I_1\mathcal O_Y\), whose reduction is the
irreducible divisor \(E=\Proj\bigl(\bigoplus_{m}I_m/I_{m+1}\bigr)\), a weighted
projective bundle over \(Z\).  Since \(q\) is finite, \(F=q(E)\) is a prime
divisor of \(W\), and \(f(F)=\pi(Z)\) has codimension at least two in \(X\).
Thus \(F\) is \(f\)-exceptional.  Conversely, if \(D\subset W\) is an
\(f\)-exceptional prime divisor, a prime divisor \(D'\subset Y\) lies over
\(D\) since \(q\) is finite and surjective, and
\(g(D')\subset\pi^{-1}(f(D))\) has codimension at least two in \(V\).  Hence
\(D'\) is \(g\)-exceptional, so \(D'=E\) and \(D=F\).
\end{proof}

\begin{proposition}\label{prop:hb-valuation-proj}
When \(D_V=p\), the model \(f:W\to X\) is crepant.  If we denote \(\nu_F=\ord_F\), then
\[
  I_m\cap R=\aideal_m(\nu_F)
  :=\{h\in R:\nu_F(h)\ge m\}
\]
for every \(m\ge0\), and
\[
  W\simeq
  \Proj_X\bigoplus_{m\ge0}\aideal_m(\nu_F)\ttt^m.
\]
\end{proposition}

\begin{proof}
By Lemma~\ref{lem:hb-support}, the relative canonical divisor is supported on
\(F\).  Proposition~\ref{prop:hb-coefficient} gives discrepancy zero, hence
\(K_W=f^*K_X\).

The only height-one point of \(Y\) above the generic point of \(F\) is the
generic point of \(E\).  Lemma~\ref{lem:hb-tail-dvr} gives ramification index
one, so
\[
  \ord_E(h)=\ord_F(h)
\]
for \(h\in k(X)^\times\).  For \(h\in R\subset S\), construction of the
weighted Rees algebra gives
\[
  h\in I_m\quad\Longleftrightarrow\quad\ord_E(h)\ge m.
\]
This proves the asserted equality of ideals.  Since the group fixes the variable \(\ttt\),
\[
  (\calS(V)^G)_m=(I_m\cap R)\ttt^m.
\]
\end{proof}

\subsection{Uniqueness of projective crepant model}

\begin{lemma}\label{lem:hb-factorial}
The quotient \(X\) is normal and factorial.
\end{lemma}

\begin{proof}
The invariants of a normal domain under a finite group
are normal \cite{CampbellWehlau2011ModularInvariantTheory}.  

\(\pi:V\to X\) is finite and \'etale in codimension one.
Let \(D\) be a Weil divisor on \(X\).  Since \(V\) is factorial,
\(\pi^*D=\operatorname{div}_V(h)\) for some \(h\in k(V)^\times\).  For
\(g\in C_p\), since \(\pi^*D\) is \(g\)-invariant and \(\operatorname{div}_V(gh)\sim \operatorname{div}_V(h)\), the quotient \(gh/h\) is a unit of \(k[V]\), which defines a character \(C_p\to k^\times\).  Since the character has to be trivial, \(h\) is invariant. Thus \(\pi^*\operatorname{div}_X(h)=\operatorname{div}_V(h)=\pi^*D\). It follows that \(D=\operatorname{div}_X(h)\), so \(X\) is factorial.
\end{proof}

\begin{theorem}\label{thm:hb-rigidity}
Assume that \(D_V=p\) and that \(W\) is terminal.  Let
\[
  g:Z\longrightarrow X
\]
be a projective crepant birational morphism from a normal variety.
If \(g\) is not an isomorphism, then it is isomorphic over \(X\) to
\(W\to X\).
\end{theorem}

\begin{proof}
Since \(W\) is terminal and crepant, and \(F\) is its only exceptional prime,
\(\nu_F\) is the unique exceptional divisorial valuation of discrepancy zero
over \(X\).  Indeed, realize any such valuation \(\nu\) by a prime divisor \(F'\) on a normal model \(Z\to W\). If \(F'\) is exceptional over \(W\), then its discrepancy \(a(F';W)>0\) by the terminality of \(W\), which leads to \(a(F';X)>0\). Hence \(F'\) is not exceptional over \(W\) but exceptional over \(X\), which forces \(F'=F\) and \(\nu=\nu_F\).  

Choose a \(g\)-ample Cartier divisor \(H\). We first claim that the morphism \(g\) has an
exceptional prime divisor.  Otherwise, set \(D=g_*H\). Since \(X\) is factorial, we obtain that \(D\) is Cartier and \(H=g^*D\).  Since \(g\) is not an isomorphism, it has a
positive-dimensional projective fiber: otherwise it is finite birational onto the normal variety \(X\), which leads to an isomorphism. Take a projective curve \(C\) contained in such a fiber. Then \(H\cdot C=(g^*D)\cdot C=0\), which contradicts the \(g\)-ampleness
of \(H\).

Thus \(g\) has an exceptional prime divisor. Every such \(g\)-exceptional prime divisor has discrepancy zero and therefore defines
\(\nu_F\).  Distinct prime divisors define distinct normalized divisorial valuations, so
\(g\) has exactly one exceptional prime \(F_Z\), with
\(\ord_{F_Z}=\nu_F\).  Again put \(D=g_*H\).  Since \(D\) is Cartier,
\[
  H-g^*D=aF_Z
\]
for some integer \(a\), and \(aF_Z\) is Cartier. \(a\) is nonzero by the same statement that we used to prove the existence of exceptional divisors. In particular, \(F_Z\) is \(\mathbb{Q}\)-Cartier.  

If \(a>0\), the
effective exceptional Cartier divisor \(aF_Z\) would be \(g\)-ample and
therefore \(g\)-nef. Then \(F_Z\) is a \(g\)-nef \(\mathbb{Q}\)-Cartier divisor, which contradicts the negativity lemma
\cite[3.39]{KollarMori1998Birational}, since \(F_Z\) is effective and exceptional.  

Hence \(a<0\) and
\(aF_Z\) is \(g\)-ample. Choose an integer \(\ell\gg0\) and set \(r=-a\ell>0\), so that \(-rF_Z=\ell(aF_Z)\) is \(g\)-very ample. Since \(Z\to X\) is projective, the relative Proj gives
\[
  Z\simeq
  \Proj_X\bigoplus_{m\ge0}g_*\mathcal O_Z(-mrF_Z).
\]
Since \(g_*\mathcal O_Z=\mathcal O_X\) and \(\ord_{F_Z}=\nu_F\),
\[
  g_*\mathcal O_Z(-mrF_Z)=\aideal_{mr}(\nu_F).
\]
Indeed, a rational function is a section of \(\mathcal O_Z(-mrF_Z)\) exactly when its order by \(\nu_F\) is at least \(mr\) and it is regular away from the exceptional prime divisor. Again by normality, the regularity away from \(F_Z\) is the same as the regularity on \(X\). 

Now, \(Z\) is the relative Proj of a positive Veronese subalgebra of the valuation algebra in
Proposition~\ref{prop:hb-valuation-proj}, while positive Veronese subalgebras have
the same relative Proj. Since \(Z\) is already normal, we obtain \(Z\simeq W\) over \(X\).
\end{proof}

\begin{proof}[Proof of Theorem~\ref{thm:higher-block}]
It is a combination of Proposition~\ref{prop:hb-terminal-singular} and Theorem~\ref{thm:hb-rigidity}.
\end{proof}

\begin{remark}\label{rem:V3-unique}
By Theorem~\ref{thm:all-v2} and Propostion~\ref{prop:hb-terminal-singular}, if \(D_V=p\), \(W\) is not terminal only when \(p=3\) and \(V=V_3\) after removing the trivial summands. A crepant resolution in this case is discussed in Remark~\ref{rem:V3-chart-age}, which is obtained by a crepant blowup \(\widetilde{W}\to W\). This blowup resolves the \(A_1\)-singularities of our model \(W\). Let \(F'\) be the unique exceptional prime divisor of \(\widetilde{W}\to W\), and identify \(F\) with its strict transform on \(\widetilde{W}\). Yasuda's computation shows that \(F\) and \(F'\) are universally homeomorphic to \(\mathbb{A}^1\times \mathbb{P}^1\), and \(F\cap F'\simeq \mathbb{A}^1\). Let \(\ell,\ell'\simeq \mathbb{P}^1\) be chosen fibers of \(F,F'\) as \(\mathbb{P}^1\)-fibers respectively. They are contracted by the resolution \(\widetilde{W}\to W\). Because a morphism from a complete
integral curve to $\mathbb A^1$ is constant and
$F\cap F'$ contains no complete curves, 
\[
\overline{\operatorname{NE}}(\widetilde{W}/X)
=
\mathbb R_{\geq0}[\ell]
+
\mathbb R_{\geq0}[\ell'].
\]
The similar argument as the proof of Theorem~\ref{thm:hb-rigidity} shows that \(\nu_F\) and \(\nu_{F'}\) are the only exceptional divisoral valuations of discrepancy zero over \(X\). We then obtain another projective crepant birational model
\[
W':=
  \Proj_X\bigoplus_{m\ge0}\mathfrak{a}_m 
  ( \nu_{F'} )
  \ttt^m
\]
with a unique exceptional divisor \(F'\) again by abuse of notation. 
The projective crepant birational models satisfy the following commutative diagram.
\[
\begin{tikzcd}
& \widetilde{W} \arrow[dr] \arrow[dl] & \\
W  \arrow[dr] & & W' \arrow[dl] \\
& X  &
\end{tikzcd}
\]
\end{remark}

We briefly explain that the projective crepant birational models in the diagram of Remark~\ref{rem:V3-unique} are all possible ones. In this case, consider an arbitrary non-isomorphic projective crepant birational model \(g:Z\to X\). In parallel with the proof of Theorem~\ref{thm:hb-rigidity}, we can similarly find exceptional prime divisors \(F_Z\) and \(F_Z'\) (one of them may not occur on \(Z\)), and furthermore a \(g\)-very ample Cartier divisor \(D_Z=\ell(H-g^*g_*H)=-aF_Z-bF'_Z\) (\(a,b\geq 0, (a,b)\neq (0,0)\)). 

If one of \(F_Z\) and \(F'_Z\) does not occur on \(Z\), then we can take either \(a=0\) or \(b=0\), and then obtain \(Z\simeq W'\) or \(Z\simeq W\) by passing to the relative Proj and the valuation model, as the proof of Theorem~\ref{thm:hb-rigidity}. 

If both \(F_Z\) and \(F'_Z\) occur, we obtain a birational map \(
\varphi:\widetilde{W} \dashrightarrow Z
\) that is isomorphic in codimension one, and \(\widetilde{D}\), the strict transform of \(D_Z\) on \(\widetilde{W}\). Since \(D_Z\) is \(g\)-very ample, \(\widetilde{D}\cdot \ell\geq0\) and \(\widetilde{D}\cdot \ell'\geq0\). Indeed, \(\widetilde{D}\cdot \ell=0\) or \(\widetilde{D}\cdot \ell'=0\) implies that \(\varphi\) contacts the corresponding exceptional divisor, which is a contradiction. Hence, \(\widetilde{D}\) is relatively ample for the model \(\tilde{f}:\widetilde{W}\to X\). Then by passing to a positive Veronese subalgebra if necessary, we have
\[
Z\simeq
  \Proj_X\bigoplus_{m\ge0}g_*\mathcal O_Z(mD_Z)
  \simeq
  \Proj_X\bigoplus_{m\ge0}\widetilde{f}_*\mathcal O_{\widetilde{W}}(m\widetilde{D})
  \simeq \widetilde{W}.
\]

Therefore, when \(p=3\) and \(V=V_3\), the four models \(\widetilde{W},W,W',X\) discussed in Remark~\ref{rem:V3-unique} give the full list of projective crepant birational models of \(X\). Since \(\widetilde{W}\) is the only smooth model among them (\(W'\) is singular: otherwise the discrepancy of \(F\) on it is positive), it is the unique projective crepant resolution of \(V_3/C_3\). Together with Theorem~\ref{thm:hb-rigidity} for the cases when \(W\) is terminal, we have the following corollary on the uniqueness of projective crepant resolutions.

\begin{corollary}\label{cor:unique-pcr}
   Assume that \(D_V=p\). If \(X\) admits projective crepant resolutions \(Y,Y'\), then \(Y\simeq Y'\) over \(X\).
\end{corollary}

%% file: pcyclic-sections/boundary-classification.tex
\section{The quotient singularities admitting projective crepant resolutions}\label{sec:pcyclic-classification}

\begin{proof}[Proof of Corollary~\ref{cor:boundary-classification}]
By Yasuda's \(p\)-cyclic McKay correspondence, when \(G\) acts linearly on \(V\) without pseudo-reflections, if \(V/G\) admits a projective crepant resolution, then \(D_V=p\). So Corollary~\ref{cor:unique-pcr} shows the uniqueness of the crepant resolution.

If \(p\ge5\), Theorem~\ref{thm:higher-block} shows that every nontrivial indecomposable summand of \(V\)
is \(V_2\), and 
\[
  V\simeq V_2^{\oplus p}\oplus V_1^{\oplus a}
\]
for some \(a\ge0\).  Theorem~\ref{thm:all-v2} gives a projective
crepant resolution.  

If \(p=3\), the only nontrivial indecomposable modules are \(V_2\)
and \(V_3\), and their contributions to \(D_V\) are \(1\) and \(3\),
respectively.  Hence \(D_V=3\) leaves precisely two essential cases:
\[
  V_2^{\oplus3}
  \qquad\text{or}\qquad
  V_3.
\]
The first is covered by Theorem~\ref{thm:all-v2}.  For the second one,
it is considered in Remark~\ref{rem:V3-chart-age}.

If \(p=2\), the only essential case is \(V_2^{\oplus 2}\), which is also treated in both Theorem~\ref{thm:all-v2} and \cite[Example~6.24]{Yasuda2014pCyclic}.
\end{proof}